\documentclass[12pt]{amsart}

\usepackage{psfrag}
\usepackage{color}
\usepackage{tikz}
\usetikzlibrary{matrix,arrows}
\usepackage{graphicx,graphics}
\usepackage{fullpage,amssymb,amsfonts,amsmath,amstext,amsthm,amscd,verbatim,enumerate}
\usepackage[T1]{fontenc}

\begin{document}

\newtheorem{theorem}{Theorem}[section]
\newtheorem{result}[theorem]{Result}
\newtheorem{fact}[theorem]{Fact}
\newtheorem{conjecture}[theorem]{Conjecture}
\newtheorem{lemma}[theorem]{Lemma}
\newtheorem{proposition}[theorem]{Proposition}
\newtheorem{corollary}[theorem]{Corollary}
\newtheorem{facts}[theorem]{Facts}
\newtheorem{props}[theorem]{Properties}
\newtheorem*{thmA}{Theorem A}
\newtheorem{ex}[theorem]{Example}
\theoremstyle{definition}
\newtheorem{definition}[theorem]{Definition}
\newtheorem{remark}[theorem]{Remark}
\newtheorem{example}[theorem]{Example}
\newtheorem*{defna}{Definition}

\newcommand{\notes} {\noindent \textbf{Notes.  }}
\newcommand{\note} {\noindent \textbf{Note.  }}
\newcommand{\defn} {\noindent \textbf{Definition.  }}
\newcommand{\defns} {\noindent \textbf{Definitions.  }}
\newcommand{\x}{{\bf x}}
\newcommand{\z}{{\bf z}}
\newcommand{\B}{{\bf b}}
\newcommand{\V}{{\bf v}}
\newcommand{\T}{\mathbb{T}}
\newcommand{\Z}{\mathbb{Z}}
\newcommand{\Hp}{\mathbb{H}}
\newcommand{\D}{\mathbb{D}}
\newcommand{\R}{\mathbb{R}}
\newcommand{\N}{\mathbb{N}}
\renewcommand{\B}{\mathbb{B}}
\newcommand{\C}{\mathbb{C}}
\newcommand{\ft}{\widetilde{f}}
\newcommand{\dt}{{\mathrm{det }\;}}
 \newcommand{\adj}{{\mathrm{adj}\;}}
 \newcommand{\0}{{\bf O}}
 \newcommand{\av}{\arrowvert}
 \newcommand{\zbar}{\overline{z}}
 \newcommand{\xbar}{\overline{X}}
 \newcommand{\htt}{\widetilde{h}}
\newcommand{\ty}{\mathcal{T}}
\renewcommand\Re{\operatorname{Re}}
\renewcommand\Im{\operatorname{Im}}
\newcommand{\tr}{\operatorname{Tr}}
\renewcommand{\skew}{\operatorname{skew}}

\newcommand{\ds}{\displaystyle}
\numberwithin{equation}{section}

\renewcommand{\theenumi}{(\roman{enumi})}
\renewcommand{\labelenumi}{\theenumi}

\title{Spiders' webs of doughnuts}

\author{A. Fletcher}
\email{fletcher@math.niu.edu}
\address{Department of Mathematical Sciences, Northern Illinois University, Dekalb, IL 60115, USA}

\author{D. Stoertz}
\email{dstoertz@cord.edu}

\subjclass[2010]{Primary 37F10; Secondary 30C65, 30D05}
\thanks{The first named author was supported by a grant from the Simons Foundation (\#352034, Alastair Fletcher).}

\begin{abstract}
If $f:\R^3 \to \R^3$ is a uniformly quasiregular mapping with Julia set $J(f)$ a genus $g$ Cantor set, for $g\geq 1$, then for any linearizer $L$ at any repelling periodic point of $f$, the fast escaping set $A(L)$ consists of a spiders' web structure containing embedded genus $g$ tori on any sufficiently large scale. In other words, $A(L)$ contains a spiders' web of doughnuts. This type of structure is specific to higher dimensions, and cannot happen for the fast escaping set of a transcendental entire function in the plane. We also show that if $f:\R^n \to \R^n$ is uqr, for $n\geq 2$ and $J(f)$ is a Cantor set, then every periodic point is in $J(f)$ and is repelling.
\end{abstract}

\maketitle

\section{Introduction}

\subsection{Background} 
Ever since Mandelbrot observed that complicated behavior could occur in the dynamics of the family of quadratic polynomials, there has been interest in the topological structures that can appear in dynamically interesting sets. In fact, it remains perhaps the greatest challenge in rational dynamics to determine whether the Mandelbrot set is locally connected. Of more relevance to this paper, however, are the structures that can occur in the fast escaping set in transcendental dynamics. 

Recall that the fast escaping set $A(f)$ of a transcendental entire function $f$ consists of points which, under iteration, escape to infinity as fast as possible commensurate with the growth of the function. This is a subset of the escaping set $I(f)$, which consists of points whose orbits diverge to infinity. Importantly, the boundary of the fast escaping set coincides with the Julia set $J(f)$. This is the set of chaotic behavior of the iterates of $f$. Typically the Julia set is very complicated, so in computer graphics relations such as equating the Julia set with the boundary of the escaping set are used.

The fast escaping set was first systematically studied by Rippon and Stallard \cite{RS} and has been the subject of much study over the last few years. Two typical classes of structures appear in fast escaping sets. First, hairs are homeomorphic copies of $[0,\infty)$. In the example $f(z) = \lambda e^z$ for $0<\lambda<1/e$ studied by Devaney \cite{Devaney}, the fast escaping set consists of a collection of hairs $\gamma((0,\infty))$ without endpoints together with some of the endpoints. Since the cross-section of $A(f)$ in a typical vertical line segment of height $2\pi$ is a Cantor set, this particular example is sometimes called a Cantor bouquet of hairs.

The second common structure is that of a spiders' web. In the plane, a spiders' web $E$ is a connected set which contains the boundaries of an increasing sequence of nested simply connected domains $(G_n)_{n=1}^{\infty}$, with $\bigcup_{n\geq 1} G_n = \C$. An incomplete selection of recent work on spiders' webs in transcendental dynamics includes \cite{Evdoridou,MBP,Osborne,Sixsmith}. 

\subsection{Quasiregular mappings}
A completely satisfactory generalization of complex dynamics to $\R^n$, $n\geq 2$, is given in the quasiregular setting. Quasiregular mappings are mappings of bounded distortion and while they are only guaranteed to be differentiable almost everywhere, they do share many important value distribution properties with holomorphic functions in the plane. Quasiconformal mappings are nothing other than injective quasiregular mappings. In the following discussion, terms will be defined in section 2.

For arbitrary quasiregular mappings, the normal family machinery from the quasiregular version of Montel's Theorem no longer applies. However, the Julia set for transcendental type mappings was defined by Bergweiler and Nicks \cite{BN} via a blowing-up property, the fast escaping set was defined and shown to be non-empty in work of Bergweiler, Drasin and the first named author \cite{BDF} and it was shown that $J(f) = \partial A(f)$ for functions that do not grow too slowly (for example, it is true for functions of positive lower order of growth) by Bergweiler, Nicks and the first named author \cite{BFN}.

One of the challenges in quasiregular dynamics is constructing mappings to which the theory applies. One way is through the idea of linearization, which we briefly outline. First, uniformly quasiregular mappings (or uqr maps for short) are a special subclass of quasiregular mappings for which there is a uniform bound on the distortion of the iterates. If the map is injective, then it is called uniformly quasiconformal, or uqc. The quasiregular version of Montel's Theorem does apply in this setting and, in particular, shows that the definitions of Fatou and Julia sets through normal families go through just as in complex dynamics. 

For uniformly quasiregular mappings, $f$ may not be differentiable at a fixed point $x_0$ and so a new approach is required. To that end, Hinkkanen, Martin and Mayer \cite{HMM} introduced generalized derivatives, which are always guaranteed to exist.
Given a repelling fixed point $x_0$ of a uqr map $f$ and a generalized derivative $\psi$, \cite{HMM} shows that there exists a transcendental type quasiregular map $L$ so that $f\circ L = L\circ \psi$. We can also consider linearizers at repelling periodic points by replacing $f$ with a suitable iterate. The dynamics of linearizers themselves can be studied, as was first done in the holomorphic setting by Mihaljevi\'{c}-Brandt and Peter \cite{MBP}.

It was observed in \cite{Fletcher} that since $\psi$ is a loxodromic repelling uqc map, it is quasiconformally conjugate to $x\mapsto 2x$ and hence there is a transcendental type quasiregular map $\Psi$ so that $f(\Psi(x)) = \Psi(2x)$. The main results from \cite{Fletcher} studied the case where $J(f)$ is a tame Cantor set, that is, a Cantor set in $\R^n$ so that there is an ambient homeomorphism mapping the Cantor set onto the standard one thirds Cantor set contained in a line. It was then shown that the corresponding map $\Psi$ has a spiders' web structure in its fast escaping set.

\subsection{Cantor sets}

In this paper, we will generalize the results from \cite{Fletcher} to all Cantor sets in $\R^3$ and comment on higher dimensional generalizations. Recall that $X\subset \R^n$ is a tame Cantor set if there is a homeomorphism $\varphi :\R^n \to \R^n$ so that $\varphi(X)$ is the standard ternary Cantor set contained in a line. Cantor sets which are not tame are called wild. Since every Cantor set in $\R^2$ is tame, see \cite{Moise}, we need to look to higher dimensions for examples of wild Cantor sets. 

The first such example is Antoine's necklace, see \cite{An}. This can be defined by taking a solid torus $T$ in $\R^3$, considering a finite collection of conformal contractions $\varphi_1,\ldots, \varphi_m$ so that $\varphi_i(T)$ forms a collection of solid tori contained in $T$ and so that $\varphi_i(T)$ and $\varphi_{i+1}(T)$ are linked via a Hopf link for $i=1,\ldots, m$ modulo $m$. The attractor set of the iterated function system generated by the $\varphi_i$ is then a wild Cantor set.

This idea can be generalized to the notion of a defining sequence for a Cantor set. A defining sequence for a Cantor set $X\subset \R^3$ is a sequence $(M_i)$ of compact $3$-manifolds with boundary such that
\begin{enumerate}[(i)]
\item each $M_i$ consists of pairwise disjoint topological cubes with handles,
\item $M_{i+1} \subset \operatorname{int}(M_i)$ for each $i$,
\item $X = \bigcap_i M_i$.
\end{enumerate}
Every Cantor set in $\R^3$ has a defining sequence, proved by Armentrout \cite{Ar} using different terminology, but the defining sequence is far from uniquely determined.
\v{Z}eljko \cite{Z} introduced the notion of the genus of a Cantor set $X$ as the infimum of the genus required in the handlebodies over all possible defining sequences for $X$. Then the genus of Antoine's necklace is $1$ and the genus of a tame Cantor set is $0$.

\subsection{Statement of results}

We generalize results of \cite{Fletcher}. There the generalized derivative $\varphi$ was conjugated to $x\mapsto 2x$, whereas here we work with the generalized derivative itself. This requires some more work, since the choice $x\mapsto 2x$ is very convenient. In particular, Lemma \ref{lem:varphi} below on controlled growth of loxodromic repelling uqc maps seems to be new and could have further applications.

For our first result, Siebert \cite{Siebert} proved that the Julia set of a uniformly quasiregular mapping is contained in the closure of the periodic points. It is still an open question whether the repelling periodic points are dense in $J(f)$. We show that this is so when $J(f)$ is a Cantor set.

\begin{theorem}
\label{thm:density}
Let $f:\R^n \to \R^n$ be a uqr map of polynomial type, and suppose that $J(f)$ is a Cantor set. Then every periodic point is in $J(f)$ and, moreover, is repelling. The repelling periodic points are thus dense in $J(f)$.
\end{theorem}

It is worth observing that while there are rational maps with $J(f)$ a Cantor set in $\overline{\C}$ and $f$ having a parabolic fixed point, these do not contradict Theorem \ref{thm:density}. This is because the hypotheses of Theorem \ref{thm:density} necessarily imply that infinity is a super-attracting fixed point of $f$. Since $F(f)$ consists of one component, it cannot be a parabolic domain.

Next, we show that if $J(f)$ is a Cantor set in $\R^3$, then its defining sequence allows conclusions to be drawn for the topological properties of $A(L)$.

\begin{theorem}
\label{thm:webs}
Let $f:\R^3 \to \R^3$ be a non-injective uqr map with $J(f)$ a Cantor set. Then if $x_0\in J(f)$ is a repelling fixed point, $\varphi \in \mathcal{D}f(x_0)$ and $L$ is a corresponding linearizer satisfying $f\circ L = L\circ \varphi$, then $A(L)$ is a spiders' web. Moreover, if $(M_i)_{i=1}^{\infty}$ is a defining sequence for $J(f)$, and $M_i^{x_0}$ is the component of $M_i$ containing $x_0$, then for all sufficiently large $i$, $A(L)$ contains continua homeomorphic to $\partial M_i^{x_0}$.
\end{theorem}

\begin{remark}
\begin{enumerate}[(i)]
\item An analogous result for repelling periodic points can be stated by replacing $f$ with an iterate.
\item We have stated this result in $\R^3$ since defining sequences are known to exist for any Cantor set. We can state an analogous result in $\R^n$, for $n\geq 4$, as long as the assumption is made that a defining sequence exists. The proof is exactly the same.
\item Since defining sequences are far from unique, for example a genus $1$ Cantor set has a defining sequence consisting of genus $g$ tori (by adding suitable small handles to each element of the genus $1$ defining sequence), this result implies that $A(L)$ must contain plenty of interesting structure.
\item The claim that there do indeed exist uqr maps in $\R^3$ with $J(f)$ a Cantor set of genus $g$, for $g\geq 2$, will be left to future work. This is true for $g=1$ by the construction of the first named author and Wu \cite{FW}.
\end{enumerate}
\end{remark}

In the case where $J(f)$ is a genus $g$ Cantor set, we have the following immediate corollary which justifies the title of this paper.

\begin{corollary}
\label{cor:1}
If $f:\R^3 \to \R^3$ is a non-injective uqr map and $J(f)$ is a genus $g$ Cantor set, then $A(L)$ is a spiders' web containing genus $g$ surfaces of arbitrarily large diameter.
\end{corollary}

The authors would like to thank Sarah Koch for motivating a study of linearizers when the Julia set is a wild Cantor set. The results in this paper will form part of the Ph.D. dissertation of the second named author.

\section{Preliminaries}

\subsection{Quasiregular maps and linearizers}

A standard reference for the foundational results of quasiregular mappings is Rickman's monograph \cite{Rickman}.
If $n\geq 2$ and $U\subset \R^n$ is a domain, a continuous mapping $f:U \rightarrow \R^{n}$  is called {\it quasiregular} if $f$ belongs to the Sobolev space $W^{1}_{n, loc}(U)$ and there exists $K \in [1, \infty)$ such that 
\begin{equation}
\label{eq2.1}
\av f'(x) \av ^{m} \leq K J_{f}(x)
\end{equation}
almost everywhere in $U$. Here $J_{f}(x)$ denotes the Jacobian determinant of $f$ at $x \in U$. The smallest constant $K \geq 1$ for which (\ref{eq2.1}) holds is called the {\it outer dilatation} $K_{O}(f)$. If $f$ is quasiregular, then we also have
\begin{equation}
\label{eq2.2}
J_{f}(x) \leq K' \inf _{\av h \av =1} \av f'(x) h \av ^{n}
\end{equation}
almost everywhere in $U$ for some $K' \in[1, \infty)$. The smallest constant $K' \geq 1$ for which (\ref{eq2.2}) holds is called the {\it inner dilatation} $K_{I}(f)$. The {\it maximal dilatation} $K=K(f)$ of $f$ is the larger of $K_{O}(f)$ and $K_{I}(f)$, and we then say that $f$ is $K$-quasiregular if $K(f)\leq K$.

A quasiregular map $f:\R^n \to \R^n$ is said to be of {\it transcendental type} if $f$ has an essential singularity at infinity. On the other hand, if $|f(x)|\to \infty$ as $|x|\to \infty$, then $f$ is said to be of {\it polynomial type}.
The composition of two quasiregular mappings is again quasiregular, although typically the distortion goes up. Denoting by $f^m$ the $m$-fold iterate of $f$, we say that $f$ is {\it uniformly $K$-quasiregular}, or $K$-uqr for short, if $K(f^m) \leq K$ for all $m\in \N$. If we do not specify the $K$, we may call the map uqr.

If $f:\R^n \to \R^n$ is uqr, then space decomposes into the {\it Julia set} $J(f)$ and the {\it Fatou set} $F(f)$ precisely as for holomorphic functions in the plane. More precisely, $x\in F(f)$ if and only if there is a neighborhood $U$ of $x$ so that the family $\{ f^m|_U : m\in \N \}$ is normal. Further, $x\in J(f)$ if and only if no such neighborhood can be found. These definitions were made by Iwaniec and Martin \cite{IM}.

If $n\geq 2$ is fixed, $x_0 \in \R^n$ and $0<r<s$, denote by $A(x_0,r,s)$ the ring domain
\[ A(x_0,r,s) = \{ x\in \R^n : r<|x-x_0| <s\}. \]
If $x_0=0$, then we write $A(r,s)$ for $A(0,r,s)$.
If $r>0$, let $S_r$ the sphere
\[ S_r  = \{ x\in \R^n : |x| = r\}.\]
If $f: \R^n \to \R^n$ is quasiregular and $r>0$, we define the {\it maximum modulus} $M(r,f)$ and the {\it minimum modulus} $m(r,f)$ by
\[ M(r,f) = \sup\{ |f(x)| : |x| = r\}, \quad \text{ and } \quad m(r,f) = \inf \{ |f(x)| : |x|=r \}\]
respectively. If $X$ is a bounded set, then we define
\[ M(X,f) = \sup \{ |f(x) | : x\in X \}, \quad \text{ and } \quad m(X,f) = \inf \{ |f(x)| : x\in X \}.\]
The iterated maximum modulus $M^m(r,f)$ is defined for $m\geq 2$ inductively by $M^m(r,f) = M^{m-1}(M(r,f) ,f)$.

If $f:\R^n \to \R^n$ is quasiregular, then the {\it escaping set} $I(f)$ is defined by
\[ I(f) = \{x\in \R^n : f^m(x) \to \infty \}.\]
For $R$ large enough that $M^m(R,f) \to \infty$, define $A_R(f)$ to be
\[ A_R(f) = \{ x\in \R^n : |f^m(x)| \geq M^m(R,f), m\in\N \}.\]
For any such $R$, the {\it fast escaping set} $A(f)$ is defined by
\[ A(f) = \{ x\in \R^n: \exists P\in \N, |f^{m+P}(x)| \geq M^m(R,f), m\in \N \}.\]

For the following definitions and results on generalized derivatives of uqr maps, we refer to \cite{HMM}.
If $f:\R^n \to \R^n$ is uqr and $x_0$ is a fixed point of $f$, then a {\it generalized derivative} $\varphi$ of $f$ is any local uniform limit of the form
\[ \varphi(x) = \lim_{k\to \infty} \frac{f(x_0 + s_k x) - f(x_0)}{s_k},\]
where $s_k \to 0$ as $k\to \infty$. Every such $\varphi$ is a uniformly quasiconformal map. There may be more than one generalized derivative at a given fixed point. The collection of generalized derivatives of $f$ at $x_0$ is called the {\it infinitesimal space} and denoted $\mathcal{D}f(x_0)$. 

\begin{definition}[Classification of fixed points]
\label{def:fix}
If $n\geq 2$ and $f:\R^n \to \R^n$ is uqr with fixed point $x_0$, then $x_0$ is said to be:
\begin{enumerate}[(i)]
\item {\it repelling} if one, and in fact all, generalized derivatives in $Df(x_0)$ are loxodromic repelling, that is, every $\varphi\in Df(x_0)$ satisfies $\varphi(0) = 0$ and $\varphi^m(x) \to \infty$ as $m\to \infty$ for $x\neq 0$;
\item {\it attracting} if one, and in fact all, generalized derivatives in $Df(x_0)$ are loxodromic attracting, that is, every $\varphi\in Df(x_0)$ satisfies $\varphi(0)=0$ and $\varphi^m(x) \to 0$ as $m\to \infty$ for $x\neq 0$;
\item {\it neutral} if the elements of $Df(x_0)$ are elliptic;
\item {\it superattracting} if $f$ is not locally injective at $x_0$.
\end{enumerate}
\end{definition}
The first two cases here agree with the topological definition of repelling and attracting fixed points; namely, a fixed point $x_0$ is called topologically repelling (respectively, attracting) if there is a neighborhood $U$ of $x_0$ so that $f$ is injective on $U$ and $f(U) \supset \overline{U}$ (respectively, $\overline{f(U)} \subset U$).

If $\varphi$ is a generalized derivative of $f$ at a repelling fixed point $x_0$, then there is a transcendental type quasiregular map $L$ satisfying $f\circ L=  L \circ \varphi$ in $\R^n$. Such a map $L$ is not unique, and we call any $L$ satisfying this functional equation a linearizer of $f$ at $x_0$. 
Note that these results are generalizations of this situation in the plane: if a holomorphic map $f$ has a repelling fixed point $z_0$ with multiplier $\lambda = f'(z_0)$ then there exists an entire Poincar\'e linearizer $L$ such that $f(L(z)) = L(\lambda z)$ for all $z\in \C$.

\subsection{Results needed for the proofs}

The following result says that global quasiconformal mappings are global quasisymmetries.

\begin{theorem}[Theorem 11.14, \cite{Heinonen}]
\label{thm:qs}
Let $n\geq 2$ and $K\geq 1$. There exists an increasing homeomorphism $\eta :[0,\infty) \to [0,\infty)$ depending only on $n$ and $K$ so that if $f:\R^n\to \R^n$ is $K$-quasiconformal, then
\[ \frac{ |f(x) - f(y)|}{|f(x) - f(z)|} \leq \eta \left ( \frac{ |x-y|}{|x-z|} \right),\]
for all $x,y,z \in \R^n$.
\end{theorem}

We will need the following result concerning the growth of polynomial type quasiregular mappings, see \cite[Lemma 2.3]{Fletcher}

\begin{lemma}
\label{lem:holderinf}
Let $h:\R^n \to \R^n$ be a $K$-quasiregular mapping of polynomial type of degree $d>K$. Then there exist $R_0>0$ and positive constants $C_1,C_2$ such that
\[ C_1 ^{ q_j ( (d/K)^{1/(n-1)} ) } |x|^{ (d/K)^{j/(n-1)} }
\leq |h^j(x)|
\leq C_2 ^{ q_j ( (dK)^{1/(n-1)} ) } |x|^{ (dK)^{j/(n-1)} }, \]
for $|x| >R_0$, where $q_j$ is the polynomial $q_j(y) = y^{j-1} + y^{j-2} + \ldots + y +1$.
\end{lemma}

We recall the following definition from \cite{BDF}. 

\begin{definition}
A set $E\subset \R^n$ is a \emph{spiders' web} if $E$ is connected and there exists a sequence $(G_m)_{m\in\N}$ of bounded topologically convex
domains (that is, the complement has no bounded components) satisfying $G_m \subset G_{m+1}$, $\partial G_m \subset E$ 
for $m\in \N$ and such that $\bigcup_{m \in \N} G_m = \R^n$.
\end{definition}

We will use the following characterization of spiders' webs in $A_R(f)$, see \cite[Proposition 6.5]{BDF}

\begin{lemma}
\label{lem:afchar}
Let $f\colon \R^n\to\R^n$ be a quasiregular mapping of transcendental type,
$R_0$ sufficiently large and $R>R_0$.
Then $A_R(f)$ is a spiders' web if and only if there exists a sequence
$(G_m)_{m=1}^{\infty}$ of bounded topologically convex domains such that, for all $m \in\N$,
\[ B(0,M^m(R,f)) \subset G_m,\]
and $G_{m+1}$ is contained in a bounded component of $\R^n \backslash f(\partial G_m)$.
\end{lemma}
 
If $A_R(f)$ is a spiders' web, then so is $A(f)$, since $A_R(f) \subset A(f)$ and every component of $A(f)$ is unbounded, see \cite{BDF}.
We require the following lemma on rescaling generalized derivatives.

\begin{lemma}
\label{lem:rescale}
Let $f:\R^n \to \R^n$ be uqr, $f(0)=0$, $f$ is locally injective at $0$ and suppose $\varphi \in \mathcal{D}f(0)$ arises as a local uniform limit of $f(s_kx)/s_k$ as $s_k\to 0$. Then if either $r_j \to \infty$ is an increasing sequence, or $r_j\to 0$ is a decreasing sequence, any locally uniform limit $\varphi_0$ of a convergent subsequence of $\varphi(r_jx)/r_j$ is also in $\mathcal{D}f(0)$.
\end{lemma}

\begin{proof}
We first assume that $r_j\to \infty$ is an increasing sequence.
Fix $R>0$. Given $\epsilon >0$, there exists $J\in \N$ so that if $j\geq J$ and $|x|\leq R$ then
\begin{equation}
\label{eq:rescale1} 
\left | \varphi_0(x) - \frac{\varphi(r_jx)}{r_j} \right | <\frac{\epsilon}{2}.
\end{equation}
Now fix $j\geq J$. Since $f(s_kx)/s_k \to \varphi$ locally uniformly on $\R^n$, then there exists $K(j) \in \N$ depending on $j$ so that $r_js_{K(j)}\to 0$ as $j\to \infty$ and  if $|x|\leq R$, that is $|r_jx| \leq r_jR$, then
\begin{equation}
\label{eq:rescale2}
\left | \frac{\varphi(r_jx)}{r_j} - \frac{f(s_kr_jx)}{s_kr_j} \right | < \frac{\epsilon}{2},
\end{equation}
for $k\geq K(j)$.
Hence \eqref{eq:rescale1} and \eqref{eq:rescale2} imply that if $j\geq J$ and $k\geq K(j)$ then
\[ \left | \varphi_0(x) - \frac{f(s_kr_jx)}{s_kr_j} \right | <\epsilon.\]
This implies that $\varphi_0 \in \mathcal{D}f(0)$.

Next, we assume that $r_j\to 0$ is a decreasing sequence. We again suppose that $R>0$ is fixed and that \eqref{eq:rescale1} holds for $|x|\leq R$ and $j\geq J$. Now fix $j\geq J$. Since $f(r_js_kx)/(r_js_k)$ converges locally uniformly to $\varphi(r_jx)/r_j$ as $k\to \infty$, find $K(j) \in \N$ large enough so that for $k\geq K(j)$,
\[ \left | \frac{\varphi(r_jx)}{r_j} - \frac{f(s_kr_jx)}{s_kr_j} \right | < \frac{\epsilon}{2} \]
for $|x|\leq R$. As above, we conclude that $\varphi_0 \in \mathcal{D}f(x_0)$.
\end{proof}

\begin{lemma}
\label{lem:varphi}
Suppose that $x_0$ is a repelling fixed point of a uqr map $f$ and that $\varphi\in \mathcal{D}f(x_0)$ is a loxodromically repelling uqc map. Then:
\begin{enumerate}[(a)]
\item there exist $C_3>1$ and $N\in \N$ so that $2|x| \leq |\varphi^N(x)| \leq C_3 |x|$ for all $x$.
\item Given $0<s<t$, there exists $C_4>1$ so that for all $m\in\N$ we have
\[ \frac{M(t,\varphi^m)}{m(s,\varphi^m)} \leq C_4.\]
\end{enumerate}
\end{lemma}

Before proving this lemma it is worth pointing out the $N$ in part (i) is necessary, as the following example shows.

\begin{example}
Define the quasiconformal maps $g_1,g_2$ in $\C$ as follows:
\[ g_1(x+iy) = \left \{ \begin{array}{cc} x+iy, & x\geq 0, \\ \frac{x}{10} +iy, & x< 0, \end{array} \right.
\quad   g_2(x+iy) = \left \{ \begin{array}{cc} -\frac{x}{10}-iy, & x\geq 0, \\ -x -iy, & x< 0, \end{array} \right.\]
and then let 
\[ g(z) = \left \{ \begin{array}{cc} g_1(z), & |z|\leq 1, \\ g_{int}(z), & 1<|z|<2, \\ g_2(z), & |z|\geq 2 \end{array} \right. \]
where $g_{int}$ is a quasiconformal interpolation of $g_1$ and $g_2$ guaranteed by Sullivan's Annulus Theorem (see the paper of Tukia and V\"ais\"al\"a \cite{TV} for the Annulus Theorem in the quasiconformal category). Now define $f(z) = g(2g^{-1}(z))$. Then $f$ is a loxodromic repelling uqc map and $f(1) = g(2) = -1/5$ and hence $m(1,f) \leq 1/5$. Thus $m(r,f) \geq r$ does not hold for this mapping.
\end{example}

\begin{proof}[Proof of Lemma \ref{lem:varphi}]
Since $\varphi$ is loxodromically repelling, we must have $M(r,\varphi)>r$ for all $r$.
We first claim that there is $C>1$ so that 
\begin{equation}
\label{eq:phi1}
M(r,\varphi) \geq Cr
\end{equation} 
for all $r$. If this is not the case, then there is a non-negative monotonic sequence $r_k$ with either $r_k\to 0$ or $r_k\to \infty$
\[ \frac{M(r_k,\varphi)}{r_k} \to 1.\]
Consider the sequence $\varphi(r_kx) / r_k$. By standard normal family results, this has a convergent subsequence with limit $\varphi_0$. By Lemma \ref{lem:rescale}, $\varphi_0 \in \mathcal{D}f(0)$. However, $M(1,\varphi_0) = 1$, which contradicts \cite[Lemma 4.4]{HMM}, namely that if one element of $\mathcal{D}f(0)$ is loxodromic repelling, then they all must be.

Next, we apply Theorem \ref{thm:qs} to $\varphi^m$. For any $m\in \N$, if $|x|=|y|=r$, $|\varphi^m(x)| = M(r,\varphi^m)$ and $|\varphi^m(y)| = m(r,\varphi^m)$, then
\begin{equation}
\label{eq:phi2}
\frac{M(r,\varphi^m)}{m(r,\varphi^m)} \leq \eta(1).
\end{equation}
Now, by induction on $m$, \eqref{eq:phi1} implies that
\[ M(r,\varphi^m) \geq C^mr\]
for all $r$. Then \eqref{eq:phi2} implies that
\[ m(r,\varphi^m) \geq \frac{C^m r}{\eta(1)}\]
for all $r$. Hence we may choose $N$ large enough so that $C^N/\eta(1) \geq 2$ to obtain the left hand inequality in (i).

For the right hand inequality, by \cite[Lemma 4.1]{HMM}, in a neighborhood of $0$, $f$ satisfies $|f(x)| \leq L|x|$ for some $L\geq 1$. Hence in a possibly smaller neighborhood of $0$, $f^N$ satisfies $|f^N(x)| \leq L^N |x|$. The process of taking the limit of $f^N(s_kx)/s_k$ as $s_k \to 0$ to find $\varphi^N$ implies that $|\varphi^N(x)| \leq L^N |x|$ for all $x\in \R^n$. This gives the right hand inequality of (i).

For (ii), choose $x$ with $|x| =t$ and $|\varphi^m(x)| = M(t,\varphi^m)$ and choose $y$ with $|y| = s$ and $|\varphi^m(y)| = m(s,\varphi^m)$. Then, again by Theorem \ref{thm:qs}, we have
\[ \frac{M(t,\varphi^m)}{m(s,\varphi^m)} = \frac{| \varphi^m(x) - \varphi^m(0)|}{|\varphi^m(y) - \varphi^m(0)|}
 \leq \eta \left ( \frac{|x-0|}{|y-0|} \right ) = \eta \left ( \frac{t}{s} \right ).\]
This proves (ii).
\end{proof}

Finally in this section, we need the folllowing topological result.

\begin{proposition}
\label{prop:cantor}
Let $n\geq 2$ and suppose $X\subset S^n$ is a Cantor set. Then $ S^n \setminus X$ is connected.
\end{proposition}

\begin{proof}
This result is presumably well-known, but we have been unable to find a direct reference for this fact. However,  Thurston \cite{Thurston} has an argument that runs as follows. By the Alexander Duality Theorem, the reduced $0$'th homology of $S^n \setminus X$ is equal to the reduced $(n-1)$'th \v{C}ech cohomology of $X$. Since $S^n \setminus X$ is open, and hence locally contractible, regular homology is fine for $S^n\setminus X$. However, \v{C}ech cohomology has to be invoked for $X$. 

Since $n\geq 2$, the $(n-1)$'th \v{C}ech cohomology of $X$ is trivial. This is given as an exercise in \cite[p.254]{ES}, but we note that it follows since there are arbitrarily fine open covers of $X$ which are disjoint and hence the corresponding nerve is zero-dimensional.
Consequently the reduced $0$'th homology of $S^n \setminus X$ is trivial which means $S^n \setminus X$ is connected.
\end{proof}

\section{Density of Repelling Periodic Points}

In this section, we will prove Theorem \ref{thm:density}. We will need the following fact about point separation in Cantor sets. This is presumably well-known, for example \cite[Lemma 3.1]{Whyburn} is the two dimensional version.

\begin{lemma}
\label{lem:smallnbhd}
Let $X$ be a Cantor set in $\R^n$, let $x\in X$, and let $\delta >0$. Then there exists a neighborhood $U$ of $x$ such that diam$(U) \leq \delta$ and $\partial U \cap X = \emptyset$.
\end{lemma}

\begin{proof}
Let $B = \overline{B(x, \delta/2)}$. Since $X\cap B$ is totally disconnected, there exists $Y\subsetneq X\cap B$ that is clopen in $X\cap B$ such that $x\in Y$. Since $Y$ is in particular closed in $X\cap B$, there exists $C\subset B$ closed such that $C\cap (X\cap B) = Y$. Now since $Y$ is open in $X$, $X\setminus Y$ is closed in $X$, and so there exists a closed set $D\subset \R^n$ with $X\setminus Y = X\cap D$. Since $\R^n$ is $T_4$, there exist open sets $U\supset C$ and $V\supset D$ in $\R^n$ such that $U \cap V = \emptyset$. Since $B$ is also $T_4$, we can choose $U$ so that $U\subset B$, ensuring that diam$(U) \leq \delta$. Since $U\cap V = \emptyset$, and $U$ and $V$ are both open, we have that $\overline{U} \cap V = \emptyset$. Since $Y\subset U$ and $X\setminus Y \subset V$, we see that $\partial U \cap X = \emptyset$.
\end{proof}

\begin{proof}[Proof of Theorem \ref{thm:density}]
Suppose that $f:\R^n \to \R^n$ is a $K$-uqr map of polynomial type and that $J(f)$ is a Cantor set. Propositon \ref{prop:cantor} and the fact that $J(f)$ is compact in $\R^n$ give that $F(f)$ is connected. Since $I(f)$ is an open connected neighborhood of infinity (see \cite{FN}), we must have that $F(f) = I(f)$.

By a result from Siebert's thesis \cite[Satz 4.3.4]{Siebert}, $J(f)$ is contained in the closure of the set of periodic points of $f$. Let $x_0$ be such a periodic point, say of period $p$. Since $F(f) = I(f)$, $x_0$ cannot be in $F(f)$, and so we have that $x_0 \in J(f)$. Henceforth, write $F = f^p$ so that $x_0$ is a fixed point of $F$.

Now let $R >0$ be such that $F$ is injective on $B(x_0, R)$. That such an $R$ exists follows from the fact that otherwise $x_0$ would be super-attracting by Definition \ref{def:fix}.

Since $J(F) = J(f)$ and $F(F) = F(f)$, we have by Lemma \ref{lem:smallnbhd} that for any small $\delta >0$, there exists a neighborhood $U_\delta$ of $x_0$ such that $\gamma_\delta = \partial U_\delta$ is contained in $I(F) \cap A(x_0, \delta/2, \delta)$. This follows since $J(F) = \partial I(F)$, and again since $I(F) = F(F)$.

We define $L(x_0,r,F) = \max\{\vert F(x) - x_0\vert:\vert x-x_0\vert =r \}$ and $l(x_0,r,F) = \min\{\vert F(x) - x_0\vert:\vert x-x_0\vert =r \}$. Furthermore, we denote by $L_m$ and $l_m$ the maximum and minimum of $\{\vert F^m(x) - x_0\vert: x \in \gamma_\delta \}$, respectively. Note first that
\[L_m \leq L(x_0, \delta, F^m), \quad l_m \geq l(x_0, \delta, F^m). \]
Additionally, note that, as long as $F^{m-1}(U_\delta) \subset B(x_0, R)$, we have that $F^m$ is $K$-quasiconformal on $U_\delta$. Hence there exists $C^* >1$ depending only on $m$ and $K$ such that
\begin{equation}\label{eq:qcgrowth}
\frac{L(x_0, \delta, F^m)}{l(x_0, \delta, F^m)} \leq C^*.
\end{equation}
Hence $L_m/l_m \leq C^*$.

Now find $C>1$ sufficiently large so that the forward orbit of any $x\in B(x_0, R/C^*)\cap I(F)$ must first pass through $A(x_0, R/(CC^*), R/C^*)$ before leaving $B(x_0, R/C^*)$. Let $\delta < R/(C(C^*)^2)$, and find $M$ minimal so that $F^M(\gamma_\delta) \cap A(x_0, R/(CC^*), R/C^*)$ is non-empty. By \eqref{eq:qcgrowth}, we then have that
\[F^M(\gamma_\delta) \subset A(x_0, R/(C(C^*)^2), R), \]
and hence that $\overline{U_\delta} \subset F^M(U_\delta)$. By the topological definition of repelling fixed points, recall the note after Definition \ref{def:fix}, $x_0$ is a repelling fixed point of $F^M$. Hence, by \cite[Proposition 4.6]{HMM}, $x_0$ is a repelling periodic point of $f$.
\end{proof}

\section{Spiders' webs}

In this section we prove Theorem \ref{thm:webs}. The method is based on that of Mihaljevi\'{c}-Brandt and Peter \cite{MBP}, but more technical difficulties need to be overcome, mainly due to linearizing to a uniformly quasiconformal map and not necessarily a dilation. We can replace $f$ by $F:=f^N$ so that the conclusions of Lemma \ref{lem:varphi} hold with $\Phi := \varphi^N$ and, moreover, the degree $d$ of $F$ is larger than $K$. Then $F:\R^3 \to \R^3$ is a $K$-uqr map and we have $F\circ L = L\circ \Phi$.

We first prove some preliminary lemmas on the growth of $L$.

\begin{lemma}
\label{lem:reggrowth}
Recall $C_3$ from Lemma \ref{lem:varphi} applied to $\Phi$.
There exists $R_1>0$ such that if $r>R_1$ then
\[ \log M(C_3^mr,L) \geq \log M(r,L) \prod _{i=1}^{m-1} \left ( \left ( \frac{d}{K} \right ) ^{1/2} + \frac{ \log C_1}{\log M(C_3^ir,L) } \right ).\]
\end{lemma}

\begin{proof}
Let $r$ be large and $y \in S_r$ such that $|L(y)| \geq |L(x)|$ for all $x \in S_r$. Let $w = L(y)$ so that $|w| = M(r,L)$. Then by the functional equation $F\circ L = L\circ \Phi$ and  Lemma \ref{lem:holderinf},
\begin{align*}
\log M(C_3r,L) & \geq \log M(r, L\circ \Phi)\\
&= \log M(r,F \circ L) \\
&= \log M( L(S_r),F )\\
& \geq \log |F(w)| \\
& \geq \log C_1 +(d/K)^{1/2} \log |w| \\
&= \log C_1 + (d/K)^{1/2} \log M(r,L).
\end{align*}
The result then follows by induction.
\end{proof}

\begin{lemma}
\label{lem:reggrowth2}
Let $\mu >1$ and recall $C_3$ from Lemma \ref{lem:varphi} applied to $\Phi$. There exist $R_2>0$ and $M\in \N$ such that for any $R>R_2$, the sequence defined by
\begin{equation}\label{eq:reg1} 
r_m = C_3^m M^m(R,L)
\end{equation}
satisfies
\[ M(r_m,L) > r_{m+1}^{\mu},\]
for $m\geq M$.
\end{lemma}

\begin{proof}
Assume $R$ is large. Set $\beta = (d/K)^{1/2}$.
With the sequence $r_m$ defined by \eqref{eq:reg1}, applying Lemma \ref{lem:reggrowth} with $r = M^m(R,L)$ yields
\begin{align*}
\log M(r_m,L) & = \log M(C_3^mM^m(R,L),L) \\
& \geq \prod _{i=0}^{m-1} \left ( \beta + \frac {\log C_1 }{\log M(C_3^i M^m (R,L) , L) } \right ) \cdot \log M( M^m(R,L) , L) \\
&\geq \left ( \beta - \frac{ \log C_1}{\log R} \right ) ^m \log M^{m+1}(R,L).
\end{align*}
Now,
\begin{align*}
\log r_{m+1}^{\mu} & = \mu \log ( C_3^{m+1} M^{m+1}(R,L) ) \\
&= \mu (m+1) \log C_3 + \mu \log M^{m+1}(R,L).
\end{align*}
Hence the result is true if
\[ \log M^{m+1}(R,L) \left (  \left ( \beta - \frac{\log C_1}{\log R} \right )^m - \mu \right ) > \mu (m+1) \log C_3.\]
Since $\mu(m+1) \log C_3$ is much smaller than $\log M^{m+1}(R,L)$ for large $m$,
this is so if we choose $R$ large enough and $m$ large enough so that $\beta^m > \mu$.
\end{proof}

\begin{lemma}
\label{lem:reggrowth3}
Let 
\begin{equation}
\label{eq:mu}
\mu > \left ( \frac{\log C_3}{\log 2} \right ) \left ( \frac{\log d + \log K}{\log d - \log K}\right ). 
\end{equation}
There exists $R_3>0$ such that for $r>R_3$ there is a continuum $\Gamma^r$ separating $S_r$ and $S_{r^{\mu}}$ such that
\[ m(\Gamma^r,L) > M(r,L).\]
\end{lemma}

\begin{proof}
There exists a neighborhood $U$ of $0$ such that $L |_{U}$ is injective, and we may assume $U \subset \B^3$. Let $\delta >0 $ be small enough that $B(x_0,\delta) \subset L(U)$. Since $J(f)$ is a Cantor set with defining sequence, there exists a continuum $\gamma_{\delta} \subset B(x_0,\delta)\cap F(f)$ so that $\gamma_{\delta}$
separates $x_0$ and infinity.
Let $\Gamma _{\delta} = L^{-1}(\gamma_{\delta}) \cap U$. Then $\Gamma_{\delta}$ is a continuum which separates $0$ from infinity. Find $0<s<t<1$ so that $\Gamma_{\delta} \subset A(s,t) := \{x:s<|x|<t\}$.

Let $r$ be large, exactly how large will be specified later. Since $\Phi^m$ is $K$-quasiconformal for all $m$, the modulus of the ring domain $\Phi^m(A(s,t))$ is uniformly bounded. 
By taking $r$ as large as we like, we can make the modulus of $A(r,r^{\mu})$ arbitrarily large. Hence by Lemma \ref{lem:varphi} we can guarantee for large $r$ there exists some $m\in \N$ so that $\Phi^m(A(s,t)) \subset A(r,r^{\mu})$. We want to estimate how many iterates of $\Phi$ we need to apply to ensure the image of $A(s,t)$ is contained in $A(r,r^{\mu})$.

Find $p_1$ minimal so that $\Phi^{p_1} (A(s,t)) \subset A(r,r^{\mu})$. By Lemma \ref{lem:varphi} (i), we can guarantee that $\Phi^{p_1}(S_s)$ meets $A(r,C_3r)$, otherwise $p_1$ would not be minimal. Hence we actually have $\Phi^{p_1}(A(s,t)) \subset A(r,C_3C_4r)$ by Lemma \ref{lem:varphi} (ii). Since this construction requires $M(t,\Phi^{p_1}) \leq C_3C_4 r$, Lemma \ref{lem:varphi} (i) again implies that
\begin{equation}
\label{eq:proof1}
2^{p_1}t \leq C_3C_4 r.
\end{equation}
Similarly, find $p_2$ maximal so that $\Phi^{p_2}(A(s,t)) \subset A(r,r^{\mu})$. As above, the fact that $p_2$ is chosen maximal means that $\Phi^{p_2} (A(s,t)) \subset A(r^{\mu}/(C_3C_4) , r^{\mu} )$. Since this requires $m(s,\Phi^{p_2}) \geq r^{\mu} / (C_3C_4)$, Lemma \ref{lem:varphi} (i) implies that
\begin{equation}
\label{eq:proof2}
C_3^{p_2}s \geq \frac{r^{\mu}}{C_3C_4}.
\end{equation}
From \eqref{eq:proof1} and \eqref{eq:proof2}, we conclude by taking logarithms that
\begin{equation}
\label{eq:proof3}
p_2 \geq p_1 \left ( \frac{\mu \log 2}{\log C_3} \right) +O(1).
\end{equation}
Here, $p_1,p_2$ depend on $r$, and we write $O(1)$ here and below for constants independent of $r$.

We clearly need $r$ large enough for these constructions to make sense, and in particular we require $A(r.C_3C_4r)$ and $A(r^{\mu} / (C_3C_4) , r^{\mu} )$ to be disjoint. This is certainly true as long as 
\[ C_3C_4 r < \frac{r^{\mu}}{C_3C_4},\]
that is if
\[ r > (C_3C_4)^{2/(\mu -1)}.\]

Next, find $R_0>0$ large enough so that Lemma \ref{lem:holderinf} may be applied to iterates of $F$. Find $j\in \N$ minimal so that $F^{j}(\gamma_{\delta}) \subset \{ x : |x| > R_0 \}$. Define $\Gamma^r := \Phi^{p_2}(\Gamma_{\delta})$. It follows immediately from the fact that $\Phi^{p_2}(A(s,t)) \subset A(r,r^{\mu})$ that $\Gamma^r$ separates $S_r$ and $S_{r^{\mu}}$.

We first estimate the minimum modulus on $\Gamma^r$:
\begin{align*}
\log m(\Gamma^r,L) & = \log m( \Gamma^r, F^{p_2} \circ L \circ \Phi^{-p_2} ) \\
& = \log m( \Gamma_{\delta}, F^{p_2} \circ L ) \\
& = \log m( \gamma_{\delta}, F^{p_2} ) \\
& \geq \log m(R_0, F^{p_2 - j} ) \\
& \geq q_{p_2 - j} ( (d/K)^{1/2} ) \log C_1 + (d/K)^{(p_2 - j)/2} \log R_0.
\end{align*}
Next,
\begin{align*}
\log M(r,L) &= \log M(r, F^{p_1} \circ L \circ \Phi^{-p_1} )\\
&= \log M( \Phi^{-p_1}(S_r) , F^{p_1} \circ L ) \\
&\leq \log M(1, F^{p_1 }\circ L ) \\
&\leq \log M(R_0, F^{p_1} )\\
& \leq q_{p_1}( (dK)^{1/2} ) \log C_2 + (dK)^{p_1 / 2} \log R_0 .
\end{align*}

We may assume $N$ is chosen large enough so that $(dK)^{1/2} \geq 2$.
Using $y^{j-1} \leq q_j(y) \leq y^j$, for $y\geq 2$, \eqref{eq:proof3} and writing $\mu_1 = \mu \log 2 / \log C_3$, these two chains of inequalities imply that we require
\[ \left ( \frac{d}{K} \right )^{(p_1\mu_1)/2 +O(1)}\log C_1 + \left ( \frac{d}{K} \right)^{(p_1\mu_1)/2 +O(1)} \log R_0 \geq (dK)^{p_1/2} \log(C_2R_0).\]
Taking logarithms, this reduces to
\[ p_1 \left ( \frac{\mu_1 \log(d/K) }{2} - \frac{ \log(dK) }{2} \right ) \geq O(1).\]
By \eqref{eq:mu}, the second factor on the left hand side is strictly positive, which means that for large enough $r$, and hence large enough $p_1$, this inequality is satisfied.
\end{proof}

Finally, we can show that $A(f)$ contains continua of the same topological type as $\gamma_{\delta}$.

\begin{proof}[Proof of Theorem \ref{thm:webs}]
Recalling Lemmas \ref{lem:holderinf}, \ref{lem:reggrowth}, \ref{lem:reggrowth2},  let $R>\max \{ R_0,R_1,R_2,R_3 \}$ and for $m \in \N$, let $r_m = C_3^m M^{m}(R,L)$ and let $\mu$ satisfy \eqref{eq:mu}. By the construction above, there is a continuum $\Gamma^{r_m}$ separating $S_{r_m}$ and $S_{r_{m}^{\mu}}$ such that
\[ m(\Gamma^{r_m}, L) >M(r_m,L).\]
We define $G_m$ to be the interior of $\Gamma^{r_m}$. Then by construction, every $G_m$ is a bounded topologically convex domain with
\[ G_m \supset \{ x \in \R^3 : |x| <r_m \} \supset \{ x \in \R^3 : |x| <M^m(R,L) \}.\]
Further, it follows from Lemma \ref{lem:reggrowth2} that
\[ m(\partial G_m, L) = m(\Gamma^{r_m},L) > M(r_m,L) > r_{m+1}^{\mu} > \max _{x \in \partial G_{m+1} } |x|,\]
and hence $G_{m+1}$ is contained in a bounded component of $\R^3 \setminus L(\partial G_m )$ and we have fulfilled the conditions of Lemma \ref{lem:afchar} for $A(L)$ to be a spiders' web.

The final statement of the theorem is immediate by choosing $\gamma_{\delta}$ appropriately.
\end{proof}

\end{document}